\documentclass{elsarticle}

\usepackage{amsmath,enumerate,amsrefs,amssymb,amsthm,nicefrac}

\newtheorem{theorem}{Theorem}[section]
\newtheorem{lemma}[theorem]{Lemma}
\newtheorem{proposition}[theorem]{Proposition}
\newtheorem{fact}[theorem]{Fact}
\newtheorem{question}[theorem]{Question}

\theoremstyle{definition}
\newtheorem{definition}[theorem]{Definition}

\newtheorem{corollary}[theorem]{Corollary}	

\newtheorem{conjecture}[theorem]{Conjecture}

\theoremstyle{remark}
\newtheorem{remark}[theorem]{Remark}

\numberwithin{equation}{section}

\newcommand{\eps}{\varepsilon}

\DeclareMathOperator{\diam}{diam}

\begin{document}

\title{Fractal Curves and Rugs of Prescribed Conformal Dimension}
\author[rvt]{Claudio A. DiMarco}
\address[rvt]{8-537, 1000 E Henrietta Rd, Mathematics Department, Monroe Community College,
Rochester, NY 14623}
\ead{cdimarco2@monroecc.edu}
\begin{keyword}
metric space, Cantor sets, Hausdorff dimension, conformal dimension, topological dimension, quasisymmetric map. \\
\MSC[2010]{Primary 28A80, 30L10; Secondary 28A78, 54F45}
\end{keyword}
\date{\today}

\begin{abstract}
We construct Jordan arcs of prescribed conformal dimension which are ``minimal for conformal dimension," meaning the Hausdorff and conformal dimensions are equal.  These curves are used to design fractal rugs, similar to Rickman's rug, that are also minimal for conformal dimension.  These fractal rugs could potentially settle a standing conjecture regarding the existence of metric spaces of prescribed topological conformal dimension.
\end{abstract}

\maketitle

\section{Introduction}

Let $(X,d)$ be a metric space.  The subscripts of dim indicate the type of dimension, and we set $\dim\varnothing=-1$ for every dimension.

Quasisymmetric maps form an interesting intermediate class lying between homeomorphisms and bi-Lipschitz maps~\cites{JH,GV}.  Topological dimension is invariant under homeomorphisms, and Hausdorff dimension is bi-Lipschitz invariant.  Conformal dimension  classifies metric spaces up to quasisymmetric equivalence~\cite{MT}:
\begin{definition}
The \textit{conformal dimension of $X$} is
\begin{equation*}
\dim_C X = \inf\{\dim_H f(X):f~\text{is quasisymmetric}\}.
\end{equation*}
\end{definition}
\noindent It is clear from the definition that conformal dimension is invariant under quasisymmetric maps, and hence under bi-Lipschitz maps.

Pansu introduced conformal dimension in 1989~\cite{PP}, and the concept has been widely studied since.  The primary applications of the theory of conformal dimension are in the study of Gromov hyperbolic spaces and their boundaries.  The boundary of a Gromov hyperbolic space admits a family of metrics which are not bi-Lipschitz equivalent, but quasisymmetrically equivalent.  Consequently, the conformal dimension of the boundary is well-defined, unlike its Hausdorff dimension~\cite{MT}.  Recent advancements involving applications of conformal dimension are exposed in \cite{MB} and \cite{BK}.  Determining the conformal dimension of the Sierpinsk\'i carpet (denoted $\dim_C SC$) is an open problem, but in \cite{KL} Keith and Laakso proved that $\dim_C SC<\dim_H SC.$  Kovalev proved a conjecture of Tyson: conformal dimension does not take values strictly between $0$ and $1$~\cite{LK}.  In \cite{HH} Hakobyan proved that if $E\subset\mathbb{R}$ is a uniformly perfect middle-interval Cantor set, then $\dim_H E=\dim_C E$ if and only if  $\dim_H E=1.$

\begin{definition}
A metric space $X$ is called \textit{minimal for conformal dimension} if $\dim_C X=\dim_H X.$
\end{definition}

In \cite{CAD} topological conformal dimension was defined; it is an adaptation of \textit{topological Hausdorff dimension} which was defined in \cite{BBE2} as
\begin{equation*}
\dim_{tH}X = \inf\{d:X~\text{has a basis}~\mathcal{U}~\text{such that}~ \dim_H \partial U \leq d-1~\text{for all}~ U\in\mathcal{U}\}.
\end{equation*}
\begin{definition}
The topological conformal dimension of $X$ is
\begin{equation*}
\dim_{tC}X = \inf\{d: X~\text{has basis}~\mathcal{U}~\text{such that} ~\dim_C \partial U \leq d-1~\text{for all}~U\in\mathcal{U}\}.
\end{equation*}
\end{definition}
There is a key difference between conformal dimension and tC-dimension.  Lower bounds for the former can be obtained through the presence of ``diffuse" families of curves, while diffuse families of surfaces provide lower bounds for the latter.  For precise statements, see Theorem 4.5 in \cite{CAD} and Proposition 4.1.3 in \cite{MT}.  While Fact 4.1 in \cite{CAD} shows $\dim_{tC}X\in \{-1,0,1\}\cup [2,\infty],$ it is unknown whether tC-dimension attains all values in $[2,\infty].$

The following conjecture was posed in \cite{CAD}:
\begin{conjecture}\label{conjecture1}
For every $d\in[2,\infty]$ there is a metric space $X$ with $\dim_{tC}X=d.$
\end{conjecture}
\noindent In this paper we provide examples of fractal spaces that could potentially settle Conjecture \ref{conjecture1}.  To this end, it seems appropriate to consider topological squares that are not quasisymmetrically equivalent to $[0,1]^2.$  A classical fractal of this kind is \textit{Rickman's rug}, which is the cartesian product of the von Koch snowflake with the standard unit interval.  In general, a \textit{fractal rug} is a product space of the form $R_d=V_d\times [0,1],$ where $V_d$ is a Jordan arc (a space homeomorphic to $[0,1]$) with $d=\dim_C V_d$.  At present, we do not have the tools necessary to determine the tC-dimensions of these fractals, but we suspect that $\dim_H R_d = \dim_{tC}R_d.$  This would be consistent with the fact that $R_d$ is minimal for conformal dimension, which follows from a result of Bishop and Tyson~\cite{MT}.  

Suppose that one prescribes $d>1$, and considers a Jordan arc $V_d$ that enjoys the minimality property $\dim_H V_d = \dim_C V_d$. In this case, it would follow that $\dim_H R_d = d+1.$  If the conjectured equality $\dim_H R_d = \dim_{tC}R_d$ were to hold, we would then have $\dim_{tC} R_d = 1+d,$ which would provide an affirmative answer to the question of existence in Conjecture \ref{conjecture1}. 

In Section \ref{rugs} we discuss fractal rugs and their dimensions in the context of Conjecture \ref{conjecture1}.  In Section \ref{arcs} we construct the Jordan arcs that are discussed in Section \ref{rugs}, which is the main result of the paper:

\begin{theorem}
For every $c\geq 1$ there is a Jordan arc $\Lambda$ with $\dim_C \Lambda=c.$
\end{theorem}

\section{Preliminaries}

The symbol $B(x,\eps)$ denotes the open ball centered at $x$ of radius $\eps.$  For $x\in\mathbb{R}^n,~|x|$ is the Euclidean modulus of $x.$  Unless otherwise stated, distance in the metric space $Y$ is denoted $d_Y.$  To discuss conformal dimension, we need the notion of quasisymmetry.  A quasisymmetric map allows for rescaling with aspect ratio control:
\begin{definition}
An embedding $f:X\rightarrow Y$ is \textit{quasisymmetric} if there is a homeomorphism $\eta:[0,\infty)\rightarrow [0,\infty)$ so that
\begin{equation*}
\frac{d_Y (f(x),f(a))}{d_Y (f(x),f(b))} \leq \eta \left( \frac{d_X(x,a)}{d_X(x,b)}\right)
\end{equation*}
for all triples $a,b,x$ of points in $X$ with $x\neq b$~\cite{MT}.
\end{definition}

Conformal dimension is defined via Hausdorff dimension.  For the latter, recall the following definition.

\begin{definition}
The \textit{p-dimensional Hausdorff measure} of $X$ is
\begin{equation*}
\mathcal{H}^p(X)=\lim_{\delta\rightarrow 0}\inf\left\{\sum_1^{\infty} (\diam E_j)^p: X\subset \bigcup_1^{\infty} E_j\text{ and }\diam E_j\leq\delta~\forall j\right\}.
\end{equation*}
The \textit{Hausdorff dimension} of $X$ is $\dim_H X=\inf\{p:\mathcal{H}^p(X)=0\}.$
\end{definition}
An interesting combination of the Hausdorff and topological dimensions called \textit{topological Hausdorff dimension} was introduced in \cite{BBE2}:
\begin{equation*}
\dim_{tH}X=\inf\{d: X~\text{has a basis}~\mathcal{U}~\text{such that}~\dim_H \partial U \leq d-1~\forall U\in \mathcal{U}\}.
\end{equation*}
In certain favorable circumstances, the Hausdorff and topological Hausdorff dimensions are additive under products.  
For any product space $X\times Y,$ we use the metric
\begin{equation*}
d((x_1,y_1),(x_2,y_2))=\max (d_X(x_1 ,x_2), d_Y(y_1, y_2)).
\end{equation*}
For sake of completeness, we include Theorem 4.21 from \cite{BBE2} and several product formulas for Hausdorff dimension (see e.g. Chapter 7 in \cite{KF}).
\begin{fact}\label{dim_H_lower}
	If $E \subset \mathbb{R}^n, ~ F \subset \mathbb{R}^m$ are Borel sets, then
	\begin{equation*}
	\dim_H (E \times F) \geq \dim_H E + \dim_H F.
	\end{equation*}
\end{fact}
\noindent Let $\overline{\dim}_H X$ be the upper box-counting dimension of $X$ (see e.g. \cite{KF}).
\begin{fact}\label{dim_H_upper}
	For any sets $E \subset \mathbb{R}^n$ and $F \subset \mathbb{R}^m$
	\begin{equation*}
	\dim_H (E \times F) \leq \dim_H E + \overline{\dim}_B F.
	\end{equation*}
\end{fact}
\noindent We call a Cantor set in $[0,1]$ \textit{uniform} if it is constructed in the same way as the usual middle-thirds example, allowing for any scaling factor $0 < r < \nicefrac{1}{2}$. Since uniform Cantor sets have equal Hausdorff and upper box dimensions, Facts \ref{dim_H_lower} and \ref{dim_H_upper} yield the following formula. 
\begin{fact}
	If $F \subset \mathbb{R}$ is a uniform Cantor set, then for any $E \subset \mathbb{R}^n$
	\begin{equation} \label{dim_H_additivity_Cantor}
	\dim_H (E \times F) = \dim_H E + \dim_H F
	\end{equation}
\end{fact}
\noindent In light of Facts \ref{dim_H_lower} and \ref{dim_H_upper}, we observe the following convenient additivity property. 
\begin{fact}\label{dim_H_additive_general}
	If $X\subset \mathbb{R}^n$ and $Y\subset \mathbb{R}^m$ are Borel sets with $\dim_H X = \overline{\dim}_B X$,
	\begin{equation}
\dim_H (X\times Y) = \dim_H X + \dim_H Y. 
	\end{equation}
\end{fact}
\noindent The condition $\dim_H X = \overline{\dim}_B X$ holds for a wide variety of spaces.

\begin{theorem}\label{dim_H_additive}
If $X$ is a nonempty separable metric space, then
\begin{equation}\label{dim_H_additive_equality}
\dim_{tH}(X\times [0,1])=\dim_H (X\times [0,1]) =\dim_H X+1.
\end{equation}
In particular, for any value $c>2,$ $R=X+1$ can be chosen such that $\dim_{tH} R = c.$
\end{theorem}
\noindent The first inequality in \eqref{dim_H_additive_equality} is due to Balka, Buczolich, and Elekes~\cite{BBE2}.  The second inequality is a generalization of Fact \ref{dim_H_upper}, which is a well-known result.

Hausdorff dimension is invariant under \textit{bi-Lipschitz maps}.
\begin{definition}
An embedding $f$ is $L$-\textit{bi-Lipschitz} if both $f$ and $f^{-1}$ are $L$-Lipschitz, and we say $f$ is \textit{bi-Lipschitz} if it is $L$-bi-Lipschitz for some $L$.
\end{definition}
\noindent Every bi-Lipschitz map is quasisymmetric, but not every quasisymmetric map is bi-Lipschitz.

We are now prepared to define conformal dimension, which measures the distortion of Hausdorff dimension by quasisymmetric maps.
\begin{definition}\label{def_conf_dim} The \textit{conformal dimension of $X$} is
\begin{equation*}
\dim_C X = \inf\{\dim_H f(X):f~is~quasisymmetric\}.
\end{equation*}
\end{definition}
\noindent In case $\dim_C X = \dim_H X$ we say that $X$ is \textit{minimal for conformal dimension}.  Bishop and Tyson proved that for every compact set $Y\subset \mathbb{R}^n,$ the space $Z = Y\times [0,1]$ is minimal for conformal dimension~\cite{MT}.  The following string of inequalities is a useful tool for determining dimensions.  The first two comprise Proposition 2.2 in \cite{CAD}, while the third is evident considering Definition \ref{def_conf_dim}.

\begin{proposition}\label{my_prop}
If $X$ is a metric space, then
\begin{equation*}
\dim_t X\leq \dim_{tC} X\leq \dim_C X \leq \dim_H X.
\end{equation*}
\end{proposition}

A \textit{Jordan arc} is an arc of a Jordan curve; that is, a homeomorphic image of $[0,1]$ with the usual topology.

Finally, we will need the notion of uniform perfectness of a metric space.  In some sense, this condition eliminates the possibility of ``large gaps" in a space.  The following definition can be found for example in \cite{JH}.  

\begin{definition}
A metric space $X$ is called \textit{uniformly perfect} if there is a constant $C \geq 1$ so that for each $x \in X$ and for each $r > 0$ the set $B(x, r) \setminus B(x, \nicefrac{r}{C})$ is nonempty whenever the set $X \setminus B(x, r)$ is nonempty.  (For the sake of definiteness, assume here that the balls are open.)
\end{definition}

\section{Rickman's Rug}\label{rugs}

Let $\eps\in (0,1)$.  The \textit{snowflake mapping}  $([0,1],|\cdot|)\rightarrow ([0,1],|\cdot|^{\eps})$ is the identity mapping.  It is quasisymmetric~\cite{JH}, and we write $[0,1]^{\eps}$ for the target space.  It is readily seen that $\dim_H \left([0,1]^{\eps}\right)=\eps^{-1}.$  Regardless of the choice of $\eps\in(0,1),$ one has $\dim_C \left([0,1]^{\eps}\right) = 1$ since the inverse of a quasisymmetric map is again quasisymmetric.  Equivalently, one can obtain the metric space $[0,1]^{\eps}$ by choosing an appropriate scaling factor and following the construction of the classical von Koch snowflake.  From this point forward, when the value $\eps\in (0,1)$ is unimportant for our discussion, we will write $V=[0,1]^{\eps}$ and refer to $R=V\times [0,1]$ as \textit{Rickman's rug}.  We use the term \textit{fractal rug} for a product space of the form $R_d=V_d\times [0,1],$ where $V_d$ is a Jordan arc with $d=\dim_C V_d,~ d\geq 1$.  As usual, this product is equipped with the metric
\begin{equation*}
d((x_1,y_1),(x_2,y_2))=\max (|x_1 - x_2|^{\eps}, |y_1 - y_2|).
\end{equation*}
The case $\eps=\frac{\ln(3)}{\ln(4)}$ corresponds to the aforementioned von Koch snowflake curve.

Since $R$ is homeomorphic to $[0,1]^2,~\dim_t R=2.$  Tukia proved that $R$ is not quasisymmetrically equivalent to $[0,1]^2$~\cite{MT}.  In fact,  Example 4.1.9 in \cite{MT} shows that $R$ is minimal for conformal dimension, meaning $\dim_C R = \dim_H R = 1+\eps^{-1}$, where the last equality follows from Theorem 4.2 in \cite{BBE2}.  We can compute the tH and tC dimensions of $R.$  Here is a simple way to compute the tC-dimension of $R.$
\begin{fact}\label{tcrug}
$\dim_{tC} R=2.$
\end{fact}
\begin{proof}
Since $V$ is a Jordan arc, Theorem 3.7 in \cite{CAD} implies $\dim_{tC} R\leq 2.$ The reverse inequality holds since $2=\dim_t R\leq \dim_{tC} R$ by Proposition \ref{my_prop}.
\end{proof}

It is not clear how to compute the topological conformal dimension of more general fractal rugs.  The difficulty in determining $\dim_{tC} R_d$ lies in giving a non-trivial lower bound.  Theorem 3.7 in \cite{CAD} yields the upper bound $\dim_{tC} R_d \leq d+1,$ but a lower bound takes into account the conformal dimension of the boundary of an arbitrary open subset of $R_d,$ which can be quite bizarre.

In view of Fact \ref{tcrug}, Rickman's rug cannot be used to answer Conjecture \ref{conjecture1}.  To accomplish that goal, one needs a more general construction.  One approach is to try to compute $\dim_{tC} R_d$ for $d>1,$ but in order to do this, one first needs to construct $V_d$ with $d>1.$  The idea of the following conjecture is to prescribe a number $c\geq 1,$ then use Theorem~\ref{lemma1} to obtain $V_{c-1}$ and ultimately show that $\dim_{tC} R_{c-1}=c.$
\begin{conjecture}\label{tc_rug_conjecture}
For any $c\geq 1$ there is a Jordan arc $V_{c-1}$ such that $\dim_{tC} (V_{c-1}\times [0,1])=c.$
\end{conjecture}
This conjecture seems reasonable if one hopes to prove it by showing that $\dim_H R_{c-1}=\dim_{tC} R_{c-1}.$  In particular, it would follow from Proposition \ref{my_prop} that $\dim_C R_{c-1}=\dim_H R_{c-1}=c.$

\section{Jordan Arcs of Prescribed Conformal Dimension}\label{arcs}

In this section we show that for any number $c\geq 1$ there is a Jordan arc with conformal dimension $c.$  The following is a modest yet useful remark on Cantor sets that will help us accomplish this task.

\begin{remark}\label{r1}
For any $a\in[0,\infty]$ there is a Cantor type set $K_a\subset [0,1]^n$ with $\dim_H K_a=a$ for large enough $n.$  For instance, if $N$ is the least positive integer such that $b=\frac{a}{N}<1$, let $K_b\subset [0,1]$ be a Cantor set with $\dim_H K_b=b$.  The set $K_b$ can be obtained in the following way \cite{PM}.  Let $0<r<\frac{1}{2}$ be such that $b=\frac{\ln(2)}{\ln(\frac{1}{r})}.$  Denote $I_{0,1} = [0,1],$ and let $I_{1,1}$ and $I_{1,2}$ be the intervals $[0,r]$ and $[0, 1-r],$ respectively.  We continue this process of selecting two subintervals of each already given interval.  If we have defined intervals $I_{k-1, 1}, \dots, I_{k-1, 2^{k-1}},$ we define $I_{k,1}, \dots, I_{k, 2^k}$ by deleting from the middle of each $I_{k-1, j}$ an interval of length $(1-2r)m(I_{k-1, j})=(1-2r)r^{k-1}.$  All the intervals $I_{k,j}$ thus obtained have length $r^k.$  It is well known that the limit set 
\begin{equation*}
K_b = \bigcap_{k=0}^{\infty} \bigcup_{j=1}^{2k} I_{k,j}
\end{equation*}
has Hausdorff dimension $b$ (see e.g. 4.10 in \cite{PM}).  Then $K_a=\prod_1^N K_b$ is a self-similar Cantor set, and $\dim_H K_a=\sum_1^N \dim_H K_b=a$ by Fact \ref{dim_H_additive_general}.
\end{remark}

\begin{theorem}\label{lemma1}
For every $c\geq 1$ there is a Jordan arc $\Lambda$ with $\dim_C \Lambda=c.$
\end{theorem}

For $c=1$ put $\Lambda=[0,1].$  We will need several lemmas to verify the case $c>1$ in Theorem \ref{lemma1}.  The result will be shown for any number $c=1+d,~d>0.$

In general, given a sequence of ratios $c_i \searrow 0$ with $\sum c_i < \infty,$ a Cantor set $E$ can be constructed as follows.  Begin with $[0,1]$ and remove the middle $c_1$st part to get two intervals of equal length.  Continuing this process, on the $i$th step, removing the middle $c_i$th part of each interval yields $2^i$ intervals of equal length.  Call the union of intervals resulting from the $i$th step $E_i.$  The resulting Cantor set $E = \bigcap E_i$ plays a pivotal role in Lemma $\ref{lemma2}.$  

We will also make use of Corollaries 3.3 and 5.6 in \cite{HH} to prove Lemma \ref{lemma2}, which are included here for sake of completeness.

\begin{corollary}[Hakobyan]\label{HHcor3.3}
	Suppose $E \subset \mathbb{R}$ is a middle interval Cantor set:
	\begin{enumerate}[(i)]
		\item If $E$ is uniformly perfect, then it is minimal for conformal dimension if and only if $\dim_H E = 1.$
		\item If $\dim_H E = 1,$ then $\dim_H f(E) \geq 1$ whenever $f$ extends to a quasisymmetric map of a uniformly perfect space.
	\end{enumerate}
\end{corollary}

\begin{corollary}[Hakobyan]\label{HHcor5.6}
	Suppose $E \subset \mathbb{R}$ is a set of conformal dimension 1 which supports a measure $\lambda_E$ such that for every $\varepsilon > 0,$ there is a constant $C$ so that whenever $x \in E$ and $R < \diam E$
	\begin{equation*}
	\frac{1}{C} R^{1 + \eps} \leq \lambda_E ( B(x, R) \cap E) \leq C R^{1 - \eps}.
	\end{equation*}
	Then for every Borel set $Y \subset \mathbb{R}^n$, 
	\begin{equation*}
	\dim_C (E \times Y) \geq \dim_H (E \times Y).
	\end{equation*}
\end{corollary}

\begin{lemma}\label{lemma2}
 Suppose $0<d<\infty.$  Let $E$ be the Cantor set constructed from the sequence of ratios $\{c_i\}_{i=1}^{\infty}$ where $c_i\searrow 0$ as $i\rightarrow \infty,~\sum c_i < \infty,$ and let $Y=K_d$ be the self-similar Cantor set with $\dim_H Y=d$ as in Remark \ref{r1}.  Then $\dim_C (E\times Y)=1+d.$
\end{lemma}

\begin{proof}
We will show that $E$ satisfies the conditions of Corollary \ref{HHcor5.6} and the result will follow.  First let us show that $E$ is uniformly perfect.  Since $\dim_H E=1,$ Corollary \ref{HHcor3.3} will then imply $\dim_C E=1.$  To this end, let $x\in E$ and $r>0.$  Write $B(x,r)\cap E=B(x,r)$ for the open ball.  Then for large enough $k$ there is a $k$th generation interval $I_{k,j}$, for some $j\in\{1,\dots,2^k\}$, such that $x\in I_{k,j}\subset B(x,r).$  Choose the smallest such $k.$  Then the length of $I_{k,j}$ is
\begin{equation}\label{1}
m(I_{k,j})=s_k=\frac{\prod_{i=1}^k (1-c_i)}{2^k},
\end{equation}
and $s_k<r\leq s_{k-1}.$  Say $I_{k,j}=[a,b]$ so that $a,b\in E.$  Then at least one of $|x-a|\geq \frac{s_k}{2}$ and $|x-b|\geq \frac{s_k}{2}$ holds.  Say $|x-a|\geq \frac{s_k}{2}.$  By \eqref{1},
\begin{equation}\label{2}
\frac{s_{k-1}}{s_k}=\frac{2^k\prod_1^{k-1}(1-c_i)}{2^{k-1}\prod_1^k (1-c_i)}=\frac{2}{1-c_k}\leq\frac{2}{1-\sup_i c_i}=K.
\end{equation}
Inequality \eqref{2} yields
\begin{equation}\label{3}
|x-a|\geq\frac{s_k}{2}\geq\frac{s_{k-1}}{2K}\geq\frac{r}{2K}.
\end{equation}
In fact, $a \in B(x, r) \setminus B(x, \nicefrac{r}{2K})$.  By \eqref{3} we have $a\in B(x,r)\setminus B(x,\frac{r}{4K})\neq\varnothing$, and hence $E$ is uniformly perfect.  Since $\dim_H E=1,$ Corollary 3.3 in \cite{HH} gives $\dim_C E=1.$  That is, $E$ is minimal for conformal dimension.

To satisfy Corollary \ref{HHcor5.6} it remains to show that $E$ supports a measure $\mu$ such that for every $\eps>0$ there is a constant $C$ so that whenever $x\in E$ and $r<\text{diam} E$,
\begin{align*}
\frac{r^{1+\eps}}{C}\leq\mu (B(x,r))\leq Cr^{1-\eps}.
\end{align*}
Write $E=\bigcap_k E_k$ where $E_k=\bigcup_{j=1}^{2^k} I_{k,j}$ are the intervals used to construct $E.$  Let $\mu_k$ be the probability measure supported on $E_k$ that gives equal weight to each $I_{k,j},j=1,\dots 2^k.$  Since $E$ is compact there is a subsequence $\mu_{k_i}\rightarrow \mu$ where $\mu$ is a probability measure supported on $E.$  In particular $\mu(I_{k,j})=\mu_k(I_{k,j})=\frac{1}{2^k}$ for all $k,j.$  Let $\eps>0,~x\in E$ and $0<r<\text{diam}E.$  Choose $k$ in the same manner as in the proof of uniform perfectness of $E.$  For some $j\in \{1,\dots,2^k\}$ we have $x\in I_{k,j}\subset B(x,r).$  Since $s_{k-1}\geq r$,
\begin{equation}\label{4}
2^{k-1}=\frac{\prod_{i=1}^{k-1} (1-c_i)}{s_{k-1}}\leq \frac{1}{r},
\end{equation}
so by \eqref{4}
\begin{equation}\label{5}
\mu(B(x,r))\geq\mu(I_{k,j})=\mu_{k}(I_{k,j})=\frac{1}{2^k}\geq\frac{r}{2}.
\end{equation}
By choice of $k$ it follows from \eqref{5} that at most three intervals of generation $k-1$ intersect $B(x,r),$ each with $\mu(I_{k-1,j})=\frac{1}{2^{k-1}}.$  Therefore $\mu(B(x,r))\leq 3\mu(I_{k-1,j})=\frac{3}{2^{k-1}}.$  Since $s_k<r$ it suffices to show that there is a constant $C$ such that
\begin{align*}
\frac{3}{2^{k-1}}\leq Cs_k^{1-\eps}=\frac{C\left(\prod_{i=1}^k (1-c_i)\right)^{1-\eps}}{2^{k(1-\eps)}}
\end{align*}
That is, we must show that there is $C$ such that $a_k\leq C$, where
\begin{equation}\label{6}
a_k=\frac{6(2^{-k\eps})}{\left(\prod_{i=1}^k(1-c_i)\right)^{1-\eps}}.
\end{equation}
Note that \eqref{6} implies $\frac{a_{n+1}}{a_n}=\frac{2^{-\eps}}{(1-c_{n+1})^{1-\eps}}\rightarrow 2^{-\eps}<1$ so that $\sum a_n<\infty$ and hence $a_n\rightarrow 0.$  In particular, $a_n$ is bounded so say $a_k\leq C$ for all $k.$  Finally, by inequality \ref{5} and the fact that $r \leq 1$,
\begin{equation}\label{7}
\frac{r^{1+\eps}}{2}\leq \mu(B(x,r))\leq Cr^{1-\eps},
\end{equation}
and by \eqref{7} there is a constant $K$ such that $\frac{1}{K}r^{1+\eps}\leq \mu(B(x,r)\cap E)\leq Kr^{1-\eps}.$  This shows that Corollary \ref{HHcor5.6} is satisfied so that
\begin{equation*}
\dim_C (E\times Y)\geq 1+\dim_H Y=1+d.
\end{equation*}
Since $Y$ is a product of uniform Cantor sets, Fact \ref{dim_H_additive_general} yields
\begin{align*}
\dim_C (E\times Y) &\leq \dim_H (E\times Y) \\
&= \dim_H E+\dim_H Y \\
&= 1+d.
\end{align*}
Therefore $\dim_C (E\times Y)=1+d.$
\end{proof}

In \cite{GV}, Gehring and V{\"a}is{\"a}l{\"a} constructed a quasiconformal mapping $f:\mathbb{R}^n\rightarrow\mathbb{R}^n$ which maps one $n$-dimensional Cantor set onto another.  Their construction involves a sequence of piecewise linear mappings, and we use that idea to produce a Jordan arc containing a (sufficiently large) product of Cantor sets.

\begin{lemma}\label{lemma3}
Let $E$ and $Y\subset [0,1]^n$ be as in Lemma \ref{lemma2}. There is a Jordan arc $\Lambda\subset [0,1]^{n+1}$ such that $\Lambda\supset (E\times Y).$
\end{lemma}

\begin{proof}
For each $k\in\mathbb{N}$ we will construct curves $\Gamma_k$ such that $\Gamma=\bigcup_k\Gamma_k$ and $\Lambda=\overline{\Gamma}.$  Since $Y\subset [0,1]^n$ is a product of $n$ copies of the same Cantor set, we see that $F_1=E_1\times Y_1$ is the first generation of $E\times Y,$ where $E_1=I_{1,1}\cup I_{1,2}$ and $Y_1=\bigcup_{i_1,\dots,i_n=1,2} (J_{1,i_1}\times\dots\times J_{1,i_n}).$  Then
\begin{equation}
F_1=\bigcup_{j,i_1,\dots,i_n=1,2}I_{1,j}\times (J_{1,i_1}\times\cdots\times J_{1,i_n})
\end{equation}
is a union of $t_1=2^{n+1}$ disjoint products whose sides are rectangles.  Let us say $F_1=\bigcup_{s=1}^{t_1} Q_s^1$ where dist$(Q_1^1,0)<\text{dist}(Q_2^1,0)\leq\cdots\leq\text{dist}(Q_{t_1-1}^1,0)<\text{dist}(Q_{t_1}^1,0).$  For each $s$ there are unique points $x_s^1,y_s^1\in Q_s^1$ with $|x_s^1|=\text{dist}(Q_s^1,0)$ and $|y_s^1|=\max\{|z|:z\in Q_s^1\}.$  For $s=1,\dots,t_1-1$ there is a simple curve $\gamma_s^1$ in $[0,1]^{n+1}$ from $y_s^1$ to $x_{s+1}^1.$  Since $n+1\geq 2$ we may choose these $2^{n+1}-1$ curves to be disjoint.  In dimension 2, for example, one can see in Figure \ref{figure2} that disjointness is guaranteed because each generation of $Y$ and $E$ are composed of disjoint pieces.

Parametrize these curves by first dividing the interval $[0,1]$ into $2(2^{n+1}-1)+1=2^{n+2}-1$ subintervals of equal length.  Call them
\begin{align*}
P_j^1=\left[\frac{j}{2^{n+2}-1},\frac{j+1}{2^{n+2}-1}\right],~~0\leq j\leq 2^{n+2}-2.
\end{align*}
Choose smooth parametrized curves $\gamma_1^1,\dots,\gamma_{t_1-1}^1\subset F_1^c$ for odd $j$:
\begin{align*}
\hat{\gamma}_j^1&:P_j^1\rightarrow\gamma_j^1~~\text{(see Figure \ref{figure2} for the case}~n=1).
\end{align*}
For this we call $\{P_j^1:j~\text{odd}\}$ {\it used} and $\{P_j^1:j~\text{even}\}$ {\it neglected}.  Put $\Gamma_1=\bigcup_{j=1}^{t_1-1}\gamma_j^1.$  Note that there are $2^{n+2}-1-(2^{n+1}-1)=2^{n+1}=t_1$ neglected subintervals of $[0,1]$ after this parametrization, which is the number of products in $F_1$.  Reindex $\{P_j^1:j~\text{even}\}=\{R_s^1\}_1^{t_1}$ in increasing order of distance from $0.$

\begin{center}
\begin{figure}[h]
\includegraphics[width=0.99\textwidth]{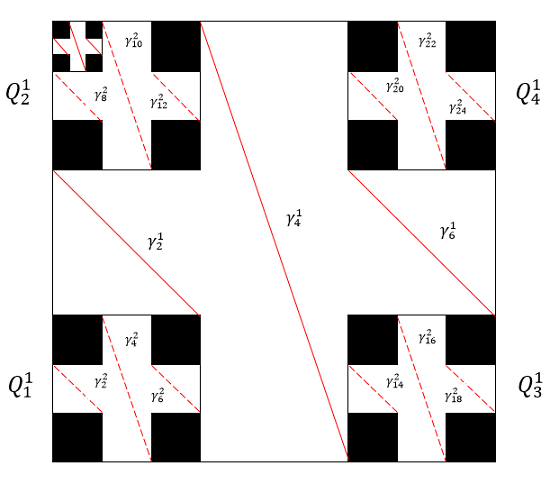} \caption[Examples of smooth curves that might comprise the first two generations of $\Gamma$]{For $n=1$ and $\dim_H Y=\frac{\ln(2)}{\ln(3)},$ these line segments are examples of smooth curves that might comprise the first two generations of $\Gamma$.  Taking the closure of the union of all such segments results in a Jordan arc with the desired conformal dimension.}\label{figure2}
\end{figure}\end{center}

For each integer $1\leq s\leq t_1$ we repeat the above path construction process for the pair $R_s^1,Q_s^1.$  Recall that $E$ is one-dimensional, and $Y$ is $n$-dimensional.  Since there are $t_1 = 2^{n+1}$ products in the first generation $\Gamma_1,$ there are $t_2 = 2^{n+1}2^{n+1} = 2^{2(n+1)}$ products in the second generation.  Each first generation product $Q_s^1$ contains $2^{n+1}$ second generation products of the form $Q_s^2.$  Within each $Q_s^1,$ we need $2^{n+1} - 1$ curves to connect the second generation products $Q_s^2.$  Since $|\{Q_s^1\}| = t_1,$ we need a total of $t_2 - t_1$ curves, so write $\Gamma_2 = \bigcup_{s = 1}^{t_2 - t_1} \gamma_s^2.$  Continuing in this fashion, it is evident that generation $k$ is composed of $t_k = 2^{k(n+1)}$ disjoint products, so that $t_k - t_{k-1}$ curves are required to connect them.  Then for each $k\in\mathbb{N},$ we have the subintervals $R_j^k,$ along with $l_k = t_k - t_{k-1}$ curves, and their union $\Gamma_k = \bigcup_{s = 1}^{l_k} \gamma_s^k.$ \\

Let $\Gamma=\bigcup_k\Gamma_k$.  It remains to show that $\overline{\Gamma}$ is a Jordan arc and that $(E\times Y)\subset \overline{\Gamma}.$  The construction of $\Gamma$ defines a function $f:D\rightarrow \Gamma$ where $D$ is dense in [0,1].  We will show that $f$ is uniformly continuous so that it extends to a continuous function $\tilde{f}:[0,1]\rightarrow\overline{\Gamma}.$  Call $x\in D$ {\it $k$-used} if $x\in \bigcup_j R_j^k$, and call $x$ {\it $k$-neglected} if $x\in \bigcap_j (R_j^k)^c.$\\

Let $\eps > 0$ and $\delta_k=m(R_1^k)$.  Take $K$ to be the smallest integer such that $\text{diam}(Q_1^K)=\dots=\text{diam}(Q_{t_K}^K)<\eps.$  Note that $\bigcup_i(\Gamma_{K+1}\cap Q_i^K) = \Gamma_{K+1}$ is composed of $l_{K+1}$ disjoint paths.  Let  $\delta'=\frac{\delta_{K+1}}{2}$ and
\begin{equation}\label{lipschitz_constant}
L_K=\max\{L_{j,k}|\hat{\gamma}_j^k ~\text{is}~ L_{j,k}\text{-Lipschitz},j~\text{even},1\leq k\leq K\}
\end{equation}
Put $\delta=\min\left\{\delta',\frac{\eps}{2 L_K}\right\}.$  If $x,y\in D$ are such that $|x-y|<\delta$ then there are three possibilities.  In any case, we must show $|f(x)-f(y)|<\eps.$\\
\begin{enumerate}
\item {\bf Both $x$ and $y$ are $(K+1)$-used.}  Since $|x - y|< \delta \leq \delta' < \delta_K$, we see that $x$ and $y$ cannot lie on opposite sides of any $K$-used subinterval.  It follows that both $x$ and $y$ must be used to parametrize curves that lie within a single $K$th generation product $Q_j^K$ for some $j.$
Therefore $|f(x)-f(y)|\leq \text{diam}(Q_j^K)=\text{diam}(Q_1^K)<\eps.$\\
\item {\bf Neither $x$ nor $y$ is $(K+1)$-used.}  Since $|x-y|<\frac{\delta_{K+1}}{2}$ it follows that either (a), (b), or (c) holds.\\
\begin{enumerate}
\item {\bf Both $x$ and $y$ are $(K+1)$-neglected.}  Note that $x,y\in R_s^{K+1}$ for some $s$.  Then $f(x),f(y)\in f(D\cap R_s^{K+1})\subset Q_s^{K+1}$ so that $|f(x)-f(y)|\leq\text{diam}(Q_1^K)<\eps.$\\
\item {\bf Both $x$ and $y$ are $k$-used for some $k\leq K$.}  Note that $x,y\in P_j^k$ for some $k\leq K.$  By \eqref{lipschitz_constant} we have
\begin{equation*}
|f(x)-f(y)|=|\hat{\gamma}_j^k(x)-\hat{\gamma}_j^k(y)|\leq  L_{j,k}|x-y|\leq  L_K\frac{\eps}{2 L_K}=\nicefrac{\eps}{2}.
\end{equation*}

\item {\bf $x$ is $k$-used for some $k\leq K$ and $y$ is $(K+1)$-neglected.}  Note that $y\in R_s^{K+1}$ for some $s.$  Without loss of generality, assume $y < x$.  Since $x$ is $k$-used, $x\in P_j^k$ for some odd $j.$  Put $P_j^k=[a,b].$  Then $f(a)$ is the corner of $Q_{j'}^k$ closest to $0.$  By construction $f(a)$ is also the corner of $Q_r^{K+1}$ closest to $0$ for some $r.$  Since $y<x$ and $y$ is $(K+1)$-neglected, $y$ is also $k$-neglected.  Then $|y-a|\leq |y-x|<\frac{1}{2}\delta_{K+1}$, and there are no $(K+1)$-used intervals between $y$ and $a.$  Therefore $f(D \cap R_s^{K+1})\subset Q_r^{K+1}$ so that
\begin{align*}
|f(y)-f(a)|\leq\text{diam}(Q_r^{K+1})<\frac{1}{2}\text{diam}(Q_1^K)<\nicefrac{\eps}{2}.
\end{align*}
Since both $a$ and $x$ are $k$-used, part (b) implies $|f(a)-f(x)|\leq\nicefrac{\eps}{2}$, so
\begin{align*}
|f(y)-f(x)|\leq |f(y)-f(a)|+|f(a)-f(x)|<\eps.
\end{align*}
\end{enumerate}

\item {\bf $x$ is $(K+1)$-used and $y$ is not $(K+1)$-used.}  Since $|x - y| < \delta_{K+1}$, it follows that $y$ is $(K+1)$-neglected.  Because $|x-y| < \delta_{K+1} < \delta_K,$ it follows that $x$ and $y$ lie in the same $K$-neglected subinterval $R_s^K$ for some $s$.  By construction  $f(D \cap R_s^K) \subset Q_i^K$ for some $i.$  Thus $|f(x)-f(y)|\leq\text{diam}(Q_i^K)<\eps.$

\end{enumerate}

So $f$ is uniformly continuous on $D$, and a continuous extension $\tilde{f}:[0,1]\rightarrow \overline{\Gamma}$ exists.  We show that $\tilde{f}$ is injective.  Let $x\neq y$ for $x,y\in [0,1].$  If $x,y\in D$ then either $f(x)$ and $f(y)$ lie on disjoint arcs so that $f(x)\neq f(y),$ or they lie on the same curve $\gamma_j^k$ in which case $f(x)\neq f(y)$ because $\hat{\gamma}_j^k$ is injective.  If $x,y\in D^c$ then there is a used interval $R_j^k$ between $x$ and $y.$  By construction, $f(x)\in Q_i^k$ and $f(y)\in Q_l^k$ for some $i\neq l$, so $f(x)\neq f(y).$  If $x\in D$ and $y\in D^c$, then there is a used interval $R_j^k$ strictly between $x$ and $y$ and the above argument implies $f(x)\neq f(y).$  Then $\tilde{f}$ is a continuous bijection whose domain is compact, so it is a homeomorphism and hence $\overline{\Gamma}$ is a Jordan arc.  To see that $(E\times Y)\subset \overline{\Gamma}$, let $z\in E\times Y$ and note that $z\in Q_{j_k}^k$ for infinitely many $k$ and $\overline{\Gamma}\cap Q_{j_k}^k\neq\varnothing$ for all $k.$  Choose $z_k\in Q_{j_k}^k\cap\overline{\Gamma}$ for each $k.$  Then $|z-z_k|\leq\text{diam}(Q_{j_k}^k)\rightarrow 0$ as $k\rightarrow\infty.$  Since $\overline{\Gamma}$ is compact, $z\in\overline{\Gamma}$, so $(E\times Y)\subset \overline{\Gamma}=\Lambda.$
\end{proof}

We now prove Theorem \ref{lemma1} with $\Lambda=\overline{\Gamma}.$

\begin{proof}[Proof of Theorem \ref{lemma1}]
By Lemmas \ref{lemma2} and \ref{lemma3} we have $\dim_C \overline{\Gamma}\geq\dim_C (E\times Y)=1+d.$  Note that $\Gamma\setminus (E\times Y)$ is a countable union of disjoint smooth curves of Hausdorff dimension 1 so that $\dim_H (\Gamma\setminus (E\times Y))=1.$  Also $\partial\Gamma\subset (E\times Y)$ so that $\overline{\Gamma}\setminus (E\times Y)=\Gamma\setminus (E\times Y).$ The stability and additivity properties of Hausdorff dimension yield
\begin{equation}\label{8}
\begin{split}
\dim_H\overline{\Gamma} &= \max\{\dim_H(\overline{\Gamma}\setminus (E\times Y)),\dim_H (E\times Y)\} \\
&= \max\{1,1+d\} \\
&= 1+d.
\end{split}
\end{equation}
It follows from \eqref{8} and the definition of conformal dimension that $\dim_C \overline{\Gamma}\leq 1+d$, and hence $\dim_C\Lambda=\dim_C \overline{\Gamma}=1+d=c.$
\end{proof}

\begin{corollary}\label{R_d_minimal}
Let $V_{c-1}$ be a Jordan arc with $\dim_C V_{c-1} = c-1,$ and let $R_{c-1} = V_{c-1} \times [0, 1].$  Then
\begin{equation*}
\dim_C R_{c-1}= \dim_H R_{c-1}=c.
\end{equation*}
\end{corollary}

Theorem \ref{lemma1} guarantees the existence of spaces $R_d = V_d \times [0,1]$, where the factor $V_d$ is a Jordan arc of prescribed conformal dimension $d$.  However, the value $\dim_{tC}R_d$ remains unknown.  Since $\dim_{tC}R_d \leq \dim_C R_d$ by Proposition \ref{my_prop}, Corollary \ref{R_d_minimal} provides a crude upper bound on $\dim_{tC}R_d$.  We do not know any non-trivial lower bounds.  Indeed, without the presence of a diffuse family of surfaces, it is difficult to determine any nontrivial lower bound on $\dim_{tC}R_d.$

\begin{question}
Determine $\dim_{tC} R_d$.
\end{question}

Topological conformal dimension and topological Hausdorff dimension are related in the following way.  For every metric space $X,$
\begin{equation}\label{th_equation}
\dim_{tC}X \leq \inf\{\dim_{tH}f(X):f~\text{quasisymmetric}\}.
\end{equation}
Question 6.4 in \cite{CAD} asks whether equality holds in \eqref{th_equation} for every $X.$  It is not clear whether the tH-dimension of $R_d$ can be lowered by quasisymmetric maps.

\begin{question}
Given $0<d<\infty,$ is there a quasisymmetric mapping $f$ such that $\dim_{tH}f(R_d)<\dim_{tH}R_d$?
\end{question}

\subsection*{Acknowledgement} This paper is based on a part of a PhD thesis written by the author under the supervision of Leonid Kovalev at Syracuse University.  The author thanks the anonymous referee for many
useful suggestions in revising this paper.

\end{document}